\documentclass[11pt]{article}
\usepackage[margin=30truemm]{geometry}
\usepackage{tikz}
\usepackage{subcaption}
\usetikzlibrary{intersections, math, calc, arrows.meta}
\usepackage{amsmath}
\usepackage{caption}
\usepackage{mathtools}
\usepackage{amssymb}
\usepackage{amsfonts}
\usepackage{amscd}
\usepackage{graphicx}
\usepackage{comment}
\usepackage{amsthm}
\usepackage{float}
\usepackage{titlesec}
\titleformat*{\section}{\large\rmfamily\bfseries}
\titleformat*{\subsection}{\large\rmfamily\bfseries}
\titleformat*{\subsubsection}{\large\rmfamily\bfseries}

\theoremstyle{definition}
\newtheorem{dfn}{Definition}[section]
\newtheorem{thm}[dfn]{Theorem}
\newtheorem{prop}[dfn]{Proposition}

\newtheorem{cor}[dfn]{Corollary}

\newtheorem{mthm}{Main Theorem}

\theoremstyle{remark}
\newtheorem{ex}[dfn]{Example}

\newtheorem{rmk}[dfn]{Remark}

\DeclareMathOperator{\MC}{MC}
\DeclareMathOperator{\MH}{MH}
\DeclareMathOperator{\OS}{OS}
\DeclareMathOperator{\rk}{rk}

\numberwithin{equation}{section}

\begin{document}

\title{\textrm{Filtered order complexes and magnitude homology of finite graded posets}}
\author{Yoh Kitajima\thanks{Department of Mathematics, 
Graduate School of Science, 
The University of Osaka, Toyonaka560-0043, JAPAN 
E-mail:yoh20010721@gmail.com}}

\maketitle

\begin{abstract}
In this paper, we study the family of subcomplexes of the order complexes of finite graded posets, 
defined via their rank function. We address three main topics.
\upshape{(1)} We describe the general topological properties of these subcomplexes in relation to magnitude homology of graded posets.
\upshape{(2)} For posets whose order complexes are homology manifolds,
we show that the homology groups of these subcomplexes agree with those of the underlying manifold except for the top dimension, 
where the homology group is a nontrivial free abelian group. 
\upshape{(3)} For shellable graded posets, we prove that each of the subcomplexes is also shellable. 
Moreover, in the case of geometric semilattices, we show that each subcomplex is homotopy equivalent to a nontrivial wedge of spheres of the same dimension.
\end{abstract}
\tableofcontents

Unless otherwise stated, all abstract simplicial complexes and posets considered in this paper are finite.

\setcounter{section}{-1}
\section{Introduction}
Magnitude of finite metric spaces was introduced by Leinster in \cite{Lei13} as an invariant of finite metric spaces.
Hepworth and Willerton in \cite{HW17} introduced magnitude homology of undirected graphs as a categorification of graph magnitude. 
Leinster and Shulman in \cite{LS21} extended this notion to generalized metric spaces and enriched categories, a class that includes finite posets. 

In relation to magnitude homology of posets, we deal the family of subcomplexes of the order complexes of finite posets.
For a finite graded poset $P$ of rank $n$ with the rank function $r$, the subcomplex $\varDelta^{(k)}(P)$ of the order complex $\varDelta(P)$ is defined as follows:
\begin{equation}
\varDelta^{(k)}(P) =\{(x_0<\cdots <x_l)\mid r(x_l)-r(x_0)\leq k\}.
\end{equation}
The family of subcomplexes $\{\varDelta^{(k)}(P)\}_{0\leq k \leq n}$ is the special case of filtered-nerve of directed graphs introduced by Di-Ivanov-Mukoseev-Zhang in \cite{DIMZ24}, 
and through this article we call this family of subcomplexes \textbf{filtered order complex} of the poset $P$.
By the definition, $\varDelta^{(0)}(P)$ is equal to $P$ itself, 
$\varDelta^{(1)}(P)$ coincides with the Hasse diagram of $P$ regarded as a 1-dimensional simplicial complex, 
and $\varDelta^{(n)}(P)=\varDelta(P)$.
Hence, order complex $\varDelta(P)$ admits the following natural filtration:
  \begin{equation}
  P= \varDelta^{(0)}(P) \subset 
  \mathop{\varDelta^{(1)}(P)}\limits_{\substack{\rotatebox{90}{$=$}\\
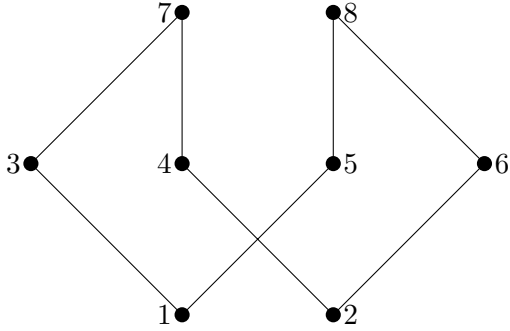
  \text{\scriptsize Hasse diagram of $P$}}}
  \subset \cdots \subset \varDelta^{(n)}(P)=\varDelta(P).
\end{equation}

For a graded poset $(P,r)$, its magnitude homology can be expressed as a relative homology of a pair of subcomplexes $(\varDelta^{(k)}(P), \varDelta^{(k-1)}(P))$ (\cite[Lemma 1.7]{DIMZ24}). 
This formula plays an important role throughout this paper, and we determine the homology and homotopy types of subcomplexes $\varDelta^{(k)}(P)$ using magnitude homology, 
in order to understand magnitude homology from a geometric perspective.

Several previous studies have adopted similar approaches to investigate magnitude homology. 
For metric spaces, a similar filtration of simplicial sets was considered by Otter \cite{Ott18}, who studied its homology as \textbf{blurred magnitude homology} and described the relationship between magnitude homology and persistence homology. 
Ivanov \cite{Iva24} introduced a similar construction and stated the analogous relationship for quasi-metric spaces.
For undirected graphs, Asao-Izumihara \cite{AI21} formulated magnitude homology as a relative homology of a pair of simplicial complexes in order to apply tools of homotopy to compute magnitude homology of graphs.
Furthermore, Tajima-Yoshinaga \cite{TY24} extended this construction to metric spaces, 
and described properties of magnitude homology of metric spaces in terms of the homotopy types of the quotient spaces associated with the corresponding pair of simplicial complexes.

This paper is divided into four sections.
In Section 1, we present the necessary background on simplicial complexes and partially ordered sets.
In Section 2, we review the magnitude homology of posets and the notion of filtered order complexes, and we discuss their basic topological properties.
%The magnitude homology groups of a graded poset $(P,r)$ can be decomposed into a direct sum of the reduced homology groups of the order complexes of open intervals in $P$ (this formula was given by Ivanov \cite[Proposition 7.2]{Iva24}).
When the magnitude homology vanishes except for its diagonal components, it is called \textit{diagonal} (this notion was introduced in \cite{HW17}), and the notion of diagonality has been one of the central themes in the study of magnitude homology.
The following theorem on the filtered order complexes of diagonal posets plays an important role in the later sections of this paper:
\begin{mthm}[Theorem \ref{diagonal order complex}]
If the magnitude homology of $P$ is diagonal, then we have the following isomorphism of the homology groups of the subcomplexes:
\begin{equation}
H_{i}(\varDelta^{(k)}(P))\cong H_{i}(\varDelta(P))\quad (i\leq k-1).
\end{equation}
The rank of the top homology group satisfies the following inequality:
\begin{equation}
\rk H_{k}(\varDelta^{(k)}(P))\leq \sum_{\substack{x\leq y\\ r(y)-r(x)=k}}|\mu(x,y)|.
\end{equation}
\end{mthm}
In Section 3, we study the properties of filtered order complexes of a poset $P$ under the assumption that $\varDelta(P)$ gives a simplicial subdivision of manifolds.
Kaneta-Yoshinaga used the undirected Hasse diagrams of face posets of triangulated manifolds with a minimum and a maximum element adjoined, to create examples of undirected graphs whose magnitude homology contains torsions in (\cite[Corollary 5.12]{KY21}).
Ivanov and Mukoseev investigated the magnitude homology of such posets in \cite{IM24}.
We establish analogous results for posets whose order complexes provide triangulations of homology manifolds, without adjoining a maximum or a minimum element, 
and show that the magnitude homology of such posets are diagonal and torsion-free. 
The proof is essentially the same strategy as that of \cite[Theorem 7.8]{IM24}.
Using this fact, we prove the following theorem about the filtered order complex of such posets.
\begin{mthm}[Theorem \ref{closed manifold}]
If $\varDelta(P)$ provides a simplicial subdivision of an $n$-dimensional homology manifold $X$, 
it holds that $H_{i}(\varDelta^{(k)}(P))\cong H_{i}(X) (i\leq k-1)$, 
and $H_{k}(\varDelta^{(k)}(P);\mathbb{Z}) (k\leq n-1)$ is a nontrivial free abelian group.
\end{mthm}
As a corollary of this theorem, we prove that for a poset whose order complex provides a simplicial subdivision of a closed manifold, 
each subcomplexes are not contractible.

In Section 4, we consider the case where the posets are shellable.
We show that when a graded poset $P$ is shellable, each subcomplex $\varDelta^{(k)}(P)$ is also shellable (Theorem \ref{shellable}). 
Moreover, we show the following theorem, which is a generalization of a result about the homotopy types of Falkman complexes of hyperplane arrangements, stated by Quillen in \cite{Qui78}.
\begin{mthm}[Theorem \ref{geometric semilattice}]
For a geometric semilattice $P$ with its rank more than $2$, 
each subcomplex $\varDelta^{(k)}(P)$ is homotopy equivalent to a nontrivial wedge of $k$-dimensional spheres.
\end{mthm}
Also, we consider the relationship between the magnitude homology of a central complex arrangement 
and the cohomology of its complement. We show that the rank of the magnitude homology of the intersection lattice of a complex central arrangement is equal to the sum of the Betti number of the complement of restrictions of the arrangement (Theorem \ref{complement}).

\section{Preliminaries}
In this section we present some background on abstract simplicial complexes and order complexes of posets.
\subsection{Abstract simplicial complex}
\begin{dfn}[\cite{Koz08}, Definition 2.1]
A finite \textbf{abstract simplicial complex} is a finite set $A$ together with a collection $\varDelta$ of subsets of $A$ 
such that if $X\in \varDelta$ and $Y\subseteq X$, then $Y\in \varDelta$.
The element $v\in A$ such that $\{v\}\in \varDelta$ is called the \textit{vertex} of $\varDelta$, 
and each $X\in \varDelta$ is called the \textit{simplex} of $\varDelta$.
\end{dfn}

\begin{dfn}[\cite{Koz08}, Definition 2.12, Definition 2.13, Definition 2.14]
Let $\varDelta$ be an abstract simplicial complex, and let $\tau$ be a simplex of $\varDelta$. 
The \textbf{deletion}, \textbf{link}, \textbf{star} of $\tau$ is the abstract simplicial subcomplexes of $\varDelta$, 
denoted by $\mathrm{dl}_{\varDelta}(\tau)$ defined by
\begin{gather}
\mathrm{dl}_{\varDelta}(\tau):=\{\sigma\in \varDelta \mid \sigma \nsupseteq \tau\}\\
\mathrm{lk}_{\varDelta}(\tau):=\{\sigma\in \varDelta \mid \sigma \cap \tau = \emptyset, \sigma\cup \tau \in \varDelta\}\\
\mathrm{star}_{\varDelta}(\tau):=\{\sigma\in \varDelta \mid \sigma\cup \tau \in \varDelta\}.
\end{gather}
Also, for two abstract simplicial complexes $\varDelta_1$ and $\varDelta_2$, 
the join of $\varDelta_1$ and $\varDelta_2$ is the abstract simplicial complex $\varDelta_1 \ast \varDelta_2$ 
with the set of vertices $V(\varDelta_1)\cup V(\varDelta_2)$, and the set of simplices
\begin{equation}
\varDelta_1 \ast \varDelta_2:=\{\sigma \in V(\varDelta_1)\cup V(\varDelta_2) \mid \sigma \cap V(\varDelta_1) \in \varDelta_1,\ \sigma \cap V(\varDelta_2) \in \varDelta_2\}. 
\end{equation}
\end{dfn}

\begin{dfn}[\cite{Koz08}]
A \textbf{matroid} $M$ is an abstract simplicial complex $\mathcal{I}$ with its set of vertices $V$ such that the following property is satisfied: 
if $\sigma, \tau \in \mathcal{I}$, and $|\sigma|>|\tau|$, then there exists $x\in \sigma \setminus \tau$ such that 
$\tau\cup \{x\} \in \mathcal{I}$. The elements of $\mathcal{I}$ are called \textbf{independent sets}.
\end{dfn}

\subsection{Partially ordered sets (posets)}
\begin{dfn}[\cite{Koz08}, Definition 2.18]
A \textbf{partially ordered set}, or simply \textbf{poset}, 
$P$ is a set together with the relation $\leq$ that satisfies the following three axioms:
\begin{itemize}
\item \textit{idenpotency}: for any $x\in P$, $x\leq x$;
\item \textit{antisymmetry}: for any $x,y\in P$, if $x\leq y$ and $y\leq x$, then $x=y$;
\item \textit{transivity}: for any $x,y,z\in P$, if $x\leq y$ and $y \leq z$, then $x\leq z$;
\end{itemize}
If $x<y\in P$ and no element $u\in P$ satisfies $x<u<y$, then we say that $y$ \textbf{covers} $x$, 
denoted $x \prec y$.
\end{dfn}

\begin{dfn}[\cite{Koz08}, Definition 9.3]
Let $P$ be a finite poset. The \textbf{order complex} of $P$ is an abstract simplicial complex $\varDelta(P)$ 
whose vertices are all elements of $P$ and whose simplices are all chains of $P$. 
\end{dfn}

\begin{rmk}
If a poset $P$ has either minimum element $\hat{0}$ or maximam element $\hat{1}$, then $\varDelta(P)$ is a cone with the apex $\hat{0}$(or $\hat{1}$), hence it is contractible. 
In this case, we often consider the order complex of the \textit{proper part} of $P$, namely $\bar{P}:=P-\{\hat{0},\hat{1}\}$.
\end{rmk}

\begin{dfn}[\cite{Koz08}, Definition 10.8]
Let $P$ and $Q$ be posets. \textbf{Ordinal sum} of $P$ and $Q$ is the poset $P\oplus Q$ 
whose set of vertices is $P \sqcup Q$ and whose order relation is given by
\begin{equation}
x\leq_{P\oplus Q} y \Longleftrightarrow
\left\{
\begin{array}{lll}
\text{either} & x,y\in P, & x\leq_{P} y;\\
\text{or} & x,y\in Q, & x\leq_{Q} y;\\
\text{or} & x\in P, & y\in Q.
\end{array}
\right.
\end{equation}
\end{dfn}

\begin{prop}[\cite{Koz08}, p.157]
For arbitrary two posets $P$ and $Q$, we have the following isomorphism of abstract simplicial complexes;
\begin{equation}
\varDelta(P\oplus Q)=\varDelta(P)\ast \varDelta(Q).
\end{equation}
\end{prop}

\begin{dfn}[\cite{Sta11}]
Let $P$ be a finite poset. If every maximal chain of $P$ has the same length $n$, 
then we say that $P$ is \textbf{pure}, or \textbf{graded of rank $n$}. 
In this case there is a unique rank function $r:P\rightarrow \mathbb{Z}_{\geq 0}$ such that 
$r(x)=0$ if $x$ is a minimal element of $P$, and $r(y)=r(x)+1$ if $x\prec y$. 
If $r(x)=i$, then we say that $x$ has \textit{rank\ $i$}.
\end{dfn}

\begin{dfn}[\cite{Sta11}]
Let $P$ be a finite graded poset of rank $n$, with a rank function $r:P\rightarrow \{0,\ldots, n\}$. 
If $S\subseteq \{0,\ldots, n\}$ then define the subposet
\begin{equation}
P_{S}=\{x\in P \mid r(x)\in S\},
\end{equation}
called the subposet \textbf{$S$-rank-selected subposet} of $P$.
\end{dfn}

\begin{dfn}[\cite{Koz08}]
For an arbitrary poset $P$, define the function $\mu:P \times P \rightarrow \mathbb{Z}$ as follows:
\begin{itemize}
\item $\mu(x,x)=1$,\ for all $x\in P$;
\item $\mu(x,y)=-\sum_{x\leq z < y}\mu(x,z)$, for all $x<y \in P$;
\item $\mu(x,y)=0$, (otherwise).
\end{itemize}
The function $\mu$ is called the \textbf{M\"{o}bius function} of $P$.
\end{dfn}

\begin{prop}[\cite{Koz08}, Theorem 10.24]\label{Mobius Euler}
For any finite poset $P$ with a maximal and minimal elements, we have
\begin{equation}
\mu(\hat{0},\hat{1})=\tilde{\chi}(\varDelta(\bar{P})),
\end{equation}
where $\tilde{\chi}$ denotes the reduced Euler characteristic (Euler characteristic minus one).
\end{prop}

\begin{cor}
Let $\mu$ be the M\"{o}bius function of $P$.
For elements $x\leq y$, the following equation holds:
\begin{equation}
\mu(x,y)=\chi(\varDelta(x,y)),
\end{equation}
where $\varDelta(x,y)$ denotes the order complex of the open interval $(x,y):=\{z\mid x<z<y\}$.
\end{cor}

\section{Magnitude homology and filtered order complexes}
\subsection{Magnitude homology of quasi-metric spaces}
\begin{dfn}[\cite{Asao23B}]
A \textit{Generalized metric space} $(X,d)$ is a set $X$ with a map $d:X\times X\rightarrow \mathbb{R}_{\geq 0}\cup \{\infty\}$ that satisfy the following conditions;
\begin{itemize}
\item $d(x,x)=0\ (\forall x\in X)$,
\item $d(x,y)+d(y,z) \geq d(x,z)\ (\forall x,y,z\in X)$.
\end{itemize}
A generalized metric space $(X,d)$ is \textit{symmetric} if it satisfies that $d(x,y)=d(y,x)$ for any $x,y\in X$,
and is \textit{non-degenerate} if $d(x,y)=d(y,x)=0$ implies that $x=y$ for any $x,y\in X$.
A generalized metric space $(X,d)$ is called \textit{quasi-metric space} if it is symmetric and non-degenerate.
\end{dfn}

\begin{ex}
Let $(P,r)$ be a finite graded poset with a rank function $r:P\rightarrow \mathbb{Z}_{\geq 0}$.
We metrize $(P,r)$ by setting the quasi-metric function $d:P\times P\rightarrow \mathbb{R}_{\geq 0}\cup \{\infty\}$ as follows:
\begin{align}\label{q-metric}
d(x,y)=\begin{cases}
r(y)-r(x) & (x\leq y),\\
\infty & (\mathrm{otherwise}).
\end{cases}
\end{align}
\end{ex}

\begin{dfn}[magnitude homology of quasi-metric spaces]
Let $(X,d)$ be a quasi-metric space.
For a tuple $(x_0,\ldots,x_n)\in X^{n+1}$, its length $L(x_0,\ldots,x_n)$ is defined by 
\begin{equation}
L(x_0,\ldots,x_n):=\sum_{i=0}^{n-1}d(x_i,x_{i+1}).
\end{equation}
For $(X,d)$ and $k\in \mathbb{R}_{\geq 0}$, the chain complex $(MC^{k}_{\ast}(X),\partial_{\ast})$ is defined as follows:
\begin{equation}
\MC_{n}^{k}(X):=\mathbb{Z}\langle 
  (x_0,x_1,\ldots,x_n)\in X^{n+1} \mid x_i\neq x_{i+1},\ L(x_0,\ldots,x_n)=k
\rangle.
\end{equation}
The boundary map $\partial^{k,n}:\MC_{n}^{k}(X)\rightarrow \MC_{n-1}^{k}(X)$ is defined by
$\partial^{k,n}:=\sum_{i=0}^{n}(-1)^{i}\partial^{k,n}_{i}$, where
\begin{equation}
\partial^{k,n}_{i}(x_0,\ldots,x_n):=\begin{cases}
    (x_0,\ldots,x_{i-1},x_{i+1},\ldots,x_n) & (L(x_0,\ldots,x_{i-1},x_{i+1},\ldots,x_n)=k)\\
    0 & (\mathrm{otherwise}).
  \end{cases}
\end{equation}
The chain complex $(\MC^{k}_{\ast}(X),\partial^{k}_{\ast})$ is called the \textit{magnitude chain complex} of $(X,d)$ in \textit{grading $k$}, 
and the \textit{magnitude homology group} of $(X,d)$ in grading $k$ is the homology $\MH_{n}^{k}(X;\mathbb{Z}):=H_{n}(\MC_{\ast}^{k}(X;\mathbb{Z}),\partial)$ 
of the magnitude chain complex $(\MC_{\ast}^{k}(X), \partial^{k}_{\ast})$ in grading $k$.
%$MH_{\ast}^{\ast}(X;\mathbb{Z})$を準距離空間$X$の\textbf{マグニチュードホモロジー(magnitude homology)}と呼ぶ.
\end{dfn}

\begin{prop}[\cite{Asao23B} Definition 4.6, \cite{HR24} p.47]
Let $a,b$ be points of a quasi-metric space $(X,d)$ and  
$\MC^{\ast}_{\ast}(a,b)$ be a submodule of $\MC^{\ast}_{\ast}(X)$ generated by the tuples $(a,x_1\ldots,x_{n-1},b)\in X^{n+1}$.
Then, we have decompositions as follows:
\begin{align}
\MC^{k}_{\ast}(X;\mathbb{Z}) &\cong \bigoplus_{a,b\in X}\MC_{\ast}^{k}(a,b),\\
\MH^{k}_{\ast}(X;\mathbb{Z}) &\cong \bigoplus_{a,b\in X}\MH_{\ast}^{k}(a,b).
\end{align}
\end{prop}

\begin{dfn}
The magnitude homology of $\MH_{*}^{*}(X)$ is \textit{diagonal} if $\MH_{n}^{k}(P)$ is trivial whenever $n\neq k$.
\end{dfn}

\begin{dfn}[Magnitude homology of graded posets]
Let $(P,r)$ be a graded poset. We define the magnitude homology $\MH^{\ast}_{\ast}(P)$ of $(P,r)$ 
by the magnitude homology of the quasi-metric space induced by the rank function $r$.
\end{dfn}

\begin{prop}[cf. \cite{KY21} Corollary 5.12, \cite{Iva24} Proposition 7.2]\label{decomposition}
The magnitude homology of a graded poset $(P,r)$ decomposes as follows:
\begin{equation}\label{m decomposition}
\MH_{n}^{k}(P)\simeq \begin{dcases}
\bigoplus_{\substack{x\leq y\\r(y)-r(x)=k}}\tilde{H}_{n-2}(\varDelta(x,y)) & (k\geq 1),\\
\mathbb{Z}^{\#P} & (n=k=0).
\end{dcases}
\end{equation}
\end{prop}
\begin{proof}
A path $(x_0,\ldots,x_n)$ of finite length in $P$ forms a chain in this order, and the length of this path is $r(x_n)-r(x_0)$.
If $n \geq 2$, then we have the following isomorphism:
\begin{equation}
\MC_{n}^{k}(x_0,x_n) \cong \mathbb{Z}\langle (x_1<\cdots<x_{n-1})\mid x_0<x_1<x_{n-1}<x_n\rangle
\cong \tilde{C}_{n-2}(\varDelta(x_0,x_n)).
\end{equation}
If $n=1$, we have $\MC_{n}^{k}(x,y) \cong \mathbb{Z}\langle (x,y)\rangle$.
Also for the boundary operators of magnitude chain complex, we have
\begin{equation}
\partial(x_0< \cdots <x_n)=\sum^{n-1}_{i=1}(-1)^{i}(x_0< \cdots < x_{i-1} < x_{i+1}< \ldots < x_n).
\end{equation}
Therefore, these isomorphisms are compatible with the boundary operators,
and we obtain $\MH_{n}^{k}(x,y)\cong H_{n-2}(\varDelta(x,y))\ (n\geq 1)$.
\end{proof}

\begin{rmk}
Proposition \ref{decomposition} implies that for a graded poset, 
diagonality of magnitude homology is equivalent to Cohen-Macaulayness (see, \cite[Chapter 3]{Sta11} for example).
\end{rmk}

\begin{prop}[cf. \cite{Asao23B} Proposition 3.36, Proposition 4.8]\label{weighting}
For a graded poset $(P,r)$, we have
\begin{equation}
\sum_{n=0}^{k}(-1)^{n}\rk \MH^{k}_{n}(P)=\sum_{\substack{x,y\in P\\ r(y)-r(x)=k}}\mu(x,y).
\end{equation}
\end{prop}
\begin{proof}
Use Proposition\ref{Mobius Euler} and Proposition\ref{decomposition}.
\end{proof}

\begin{cor}\label{diagonal}
If $(P,r)$ is diagonal, we have
\begin{equation}
\rk \MH^{k}_{k}(P)=\sum_{\substack{x,y\in P\\ r(y)-r(x)=k}}|\mu(x,y)|.
\end{equation}
\end{cor}

\subsection{Filtered order complex of posets}
The following definition is the special case of "filtered nerve of digraphs" defined in \cite{DIMZ24}.
\begin{dfn}[Filtration of the order complex of a graded poset, cf. \cite{DIMZ24}, Section 1.3]
Let $(P,r)$ be a graded poset of rank $n$.
For each $0\leq k \leq n$, we define the subcomplex $\varDelta^{(k)}(P)$ of the order complex $\varDelta(P)$ of $P$ by
\begin{equation}
\varDelta^{(k)}(P):=\{(x_0<\cdots<x_l)\mid r(x_l)-r(x_0)\leq k\}.
\end{equation}
The family of subcomplexes $\{\varDelta^{(k)}(P)\}_{0\leq k \leq n}$ will be reffered to as the \textit{filtered order complexes} of $P$.
\end{dfn}

\begin{ex}\label{chain}
Let $P$ be a chain $[n]:=\{0<1<2<\cdots<n\}$ of length $n$, and define its rank function by $r(i)=i\ (0\leq i \leq n)$.
Hasse diagram of $P$ is a line graph consists of $(n+1)$ vertices, 
and the order complex $\varDelta(P)$ is a $n$-simplex.
$\varDelta^{(k)}(P)$ is contractible.

\begin{figure}[H]
  \centering
  \begin{tabular}{ccc}
    \begin{minipage}{0.3\textwidth}
      \centering

      \begin{tikzpicture}
        
        \coordinate (0) at (0,0);
        \coordinate (1) at (0,1);
        \coordinate (2) at (0,2);
        \coordinate (3) at (0,3);

        \node[below] at (0) {$0$};
        \node[right] at (1) {$1$};
        \node[right] at (2) {$2$};
        \node[above] at (3) {$3$};

        \draw (0)--(1)--(2)--(3);

        \foreach\P in {0,1,2,3} \fill[black]
        (\P)circle(0.1);
      \end{tikzpicture}
      \caption*{$\varDelta^{(1)}(P)$}
    \end{minipage}
  
    \begin{minipage}{0.3\textwidth}
      \centering
      \begin{tikzpicture}
        \coordinate (0) at (0,2);
        \coordinate (1) at (-1,0);
        \coordinate (2) at (1,0);
        \coordinate (3) at (0,-2);

        \node[above] at (0) {$0$};
        \node[left] at (1) {$1$};
        \node[right] at (2) {$2$};
        \node[below] at (3) {$3$};

        \fill[lightgray, draw=black] (0)--(1)--(2)--cycle;
        \fill[lightgray, draw=black] (1)--(2)--(3)--cycle;

        \foreach\P in {0,1,2,3} \fill[black]
        (\P)circle(0.1);
      \end{tikzpicture}
      \caption*{$\varDelta^{(2)}(P)$}
    \end{minipage}

    \begin{minipage}{0.3\textwidth}
    \begin{tikzpicture}
        \coordinate (0) at (0,2.5);
        \coordinate (1) at (-2,-1);
        \coordinate (2) at (2,-1);
        \coordinate (3) at (2,1);

        \node[above] at (0) {$0$};
        \node[left] at (1) {$1$};
        \node[right] at (2) {$2$};
        \node[right] at (3) {$3$};

        \fill[gray, draw=black] (0)--(1)--(2)--cycle;
        \fill[gray, draw=black] (0)--(2)--(3)--cycle;
        \draw[black, dashed] (1)--(3);

        \foreach\P in {0,1,2,3} \fill[black]
        (\P)circle(0.1);
    \end{tikzpicture}
    \caption*{$\varDelta^{(3)}(P)$}
  \end{minipage}
  \end{tabular}

  \caption*{Example \ref{chain}, $n=3$}
\end{figure}
\end{ex}

\begin{rmk}
It is easy to see that when $\varDelta^{(1)}(P)$ is contractible, 
subcomplexes $\varDelta^{(k)}(P)\ (k\geq 2)$ are also contractible(sequence of collapsing $k$-simplices having vertices corresponding the leaf of $\varDelta^{(1)}(P))$. 
It is an open problem whether $\varDelta^{(k)}(P)$ is contractible whenever $\varDelta^{(k-1)}(P)$ is contractible.
\end{rmk}

When considering the topological properties of filtered order complexes, 
sometimes it is useful to represent these subcomplexes as the union of the order complexes of rank selections glued together.

\begin{prop}\label{rank selection}
For a graded poset $(P,r)$ of rank $n$, we have
\begin{equation}\label{sum}
\varDelta^{(k)}(P)=\bigcup_{i=0}^{n-k}\varDelta(P_{[i,i+k]}).
\end{equation}
Also, we have the injection $H_{k}(\varDelta(P_{[i,i+k]}))\hookrightarrow H_k(\varDelta^{(k)}(P))$.
\end{prop}
\begin{proof}
We have \eqref{sum} by definitions. Injectivity follows from Mayer-Vietoris sequence.
\end{proof}

\begin{prop}\label{relative homology}
For a graded poset $(P,r)$, we have the following isomorphism:
\begin{equation}
\MH^{k}_{n}(P)\cong \begin{cases}
H_{n}(\varDelta^{(k)}(P),\varDelta^{(k-1)}(P)) & (n \geq 1)\\
\mathbb{Z}^{\#P} & (k=n=0),\\
0 & (\mathrm{otherwise}).
\end{cases}
\end{equation}
\end{prop}
\begin{proof}
If $n\geq 1$, we have the following isomorphism regarding the magnitude chain complexes of $P$:
\begin{align}
\MC^{k}_{n}(P) & = \mathbb{Z}\langle (x_0 < \cdots < x_n) \mid r(x_n)-r(x_0)=k\rangle \\
                            & = C_{n}(\varDelta^{(k)}(P),\varDelta^{(k-1)}(P))
\end{align}
It is easy to see that the boundary operators are compatible with the isomorphism.
Therefore, we obtain $\MH^{k}_{n}(P)\cong H_{n}(\varDelta^{(k)}(P),\varDelta^{(k-1)}(P))$.
\end{proof}

\begin{thm}\label{diagonal order complex}
If $P$ is diagonal, then we have the following isomorphism of the homology groups of the subcomplexes:
\begin{equation}\label{homology of Delta^k}
H_{i}(\varDelta^{(k)}(P))\cong H_{i}(\varDelta(P))\quad (0\leq i\leq k-1).
\end{equation}
The rank of the top homology group satisfies the following inequality:
\begin{equation}\label{rank}
\beta_{k}(\varDelta^{(k)}(P))\leq \sum_{\substack{x\leq y\\ r(y)-r(x)=k}}|\mu(x,y)|.
\end{equation}
Here, $\beta_{i}$ denotes the $i$-th Betti number.
\end{thm}
\begin{proof}
We have the following long exact sequence of homology groups for a pair of spaces $(\varDelta^{(k+1)}(P),\varDelta^{(k)}(P))$:
\begin{equation}\label{LES}
\dots \rightarrow H_{i}(\varDelta^{(k)}(P))\rightarrow H_{i}(\varDelta^{(k+1)}(P)) \rightarrow H_{i}(\varDelta^{(k+1)}(P),\varDelta^{(k)}(P))\rightarrow \dots.
\end{equation}
By proposition \ref{relative homology} and diagonality, we obtain $H_i(\varDelta^{(k+1)}(P))\cong H_{i}(\varDelta^{(k)}(P))$.
Applying this repeatedly, we get \eqref{homology of Delta^k}.\\
By \eqref{LES}, we have the inclusion $H_{k}(\varDelta^{(k)}(P))\hookrightarrow \MH^{k}_{k}(P)$; 
hence Proposition \ref{diagonal} yields \eqref{rank}.
\end{proof}

\begin{ex}
The following is an example of a poset which is diagonal, non-contractible and the equality of \eqref{rank} holds.
Let $P$ be a poset of rank $2$ with the Hasse diagram shown in Figure \ref{fig:P}.
The magnitude homology of $P$ and the homotopy types of $\varDelta^{(k)}(P)\ (0\leq k \leq 2)$ are as follows:
\begin{gather}
\MH^{2}_{i}(P)\cong 0\ (i=0,1,2), \quad
\MH^{1}_{i}(P)\cong \begin{cases}
\mathbb{Z}^{8} & (i=1)\\
0 & (\mathrm{otherwise}),
\end{cases} \quad
\MH^{0}_{0}(P)\cong \mathbb{Z}^{8},\\
\varDelta^{(2)}(P)\simeq \mathbb{S}^{1}, \quad
\varDelta^{(1)}(P) \simeq \mathbb{S}^1, \quad
\varDelta^{(0)}(P) \simeq \bigvee^{7}\mathbb{S}^0.
\end{gather}
The equality of \eqref{rank} holds in the case $k=2$.
\begin{figure}[H]
\centering
\begin{minipage}{0.45\textwidth}
\centering
        \begin{tikzpicture}
            \coordinate (1) at (-1,0) ;
            \coordinate (2) at (1,0) ;
            \coordinate (3) at (-3,2) ;
            \coordinate (4) at (-1,2) ;
            \coordinate (5) at (1,2) ;
            \coordinate (6) at (3,2) ;
            \coordinate (7) at (-1,4);
            \coordinate (8) at (1,4);

            \node[left] at (1) {$1$};
            \node[right] at (2) {$2$};
            \node[left] at (3) {$3$};
            \node[left] at (4) {$4$};
            \node[right] at (5) {$5$};
            \node[right] at (6) {$6$};
            \node[left] at (7) {$7$};
            \node[right] at (8) {$8$};

            \draw (1)--(3)--(7)--(4)--(2)--(6)--(8)--(5)--cycle;
            \foreach\P in {1,2,3,4,5,6,7,8} \fill[black]
            (\P)circle(0.1);
        \end{tikzpicture}
        \caption{Hasse diagram of $P$.}
        \label{fig:P}
    \end{minipage} \hspace{0.05\textwidth}
    \begin{minipage}{0.45\textwidth}
      \begin{tikzpicture}
            \coordinate (1) at (-1,1) ;
            \coordinate (2) at (1,-1) ;
            \coordinate (3) at (-3,0) ;
            \coordinate (4) at (0,-3) ;
            \coordinate (5) at (0,3) ;
            \coordinate (6) at (3,0) ;
            \coordinate (7) at (-1,-1);
            \coordinate (8) at (1,1);

            \node[above left] at (1) {$1$};
            \node[below right] at (2) {$2$};
            \node[left] at (3) {$3$};
            \node[left] at (4) {$4$};
            \node[right] at (5) {$5$};
            \node[right] at (6) {$6$};
            \node[below left] at (7) {$7$};
            \node[above right] at (8) {$8$};

            \fill[lightgray, draw=black](1)--(3)--(7)--cycle;
            \fill[lightgray, draw=black](1)--(5)--(8)--cycle;
            \fill[lightgray, draw=black](2)--(7)--(4)--cycle;
            \fill[lightgray, draw=black](2)--(6)--(8)--cycle;

            \foreach\P in {1,2,3,4,5,6,7,8} \fill[black]
            (\P)circle(0.1);
      \end{tikzpicture}
      \caption{The order complex of $P$.}
\end{minipage}
\end{figure}
\end{ex}

\section{Magnitude homology and filtered order complexes associated with manifolds}
\subsection{Magnitude homology of posets associated with manifolds}
First, we review some of the combinatorial aspects of manifolds.
Throughout this section, we deal with triangulable spaces.

\begin{dfn}
A topological space $X$ is called a \textit{homology $n$-manifold (without boundaries)} if for each point $x$ of $X$, 
the local homology group $H_{i}(X,X-x)$ vanishes if $i\neq n$ and is isomorphic to $\mathbb{Z}$ if $i=n$.
\end{dfn}

Usual topological $n$-manifolds (without boundaries) are homology $n$-manifolds(by excisions and long exact sequence of a pair of spaces).

\begin{prop}[\cite{Mun84b} Theorem 63.2]\label{link sphere}
Let $\varDelta$ be a triangulated homology $n$-manifold. 
Let $\sigma\in \varDelta$ be a $k$-simplex of $\varDelta$. 
If $\mathrm{lk}\ \sigma$ is non-empty, then it is a homology $n-k-1$ sphere.
\end{prop}

\begin{cor}[\cite{Sta11} p.310]\label{interval}
Let $P$ be a finite poset. If $\varDelta(P)$ is a homology $n$-manifold, 
then any non-empty open interval is a homology sphere.
\end{cor}
\begin{proof}
For a non-empty open interval $(x,y)$ in $P$, 
take a maximal chain of $P_{\leq x}$ by $x_0<\cdots<x_i<x$, and a maxial chain in $P_{\geq y}$ by $y<y_{1}<\cdots < y_{j}$.
Let $\sigma$ be a simplex of $\varDelta(P)$ represented by a chain $x_0<\cdots < x_i < y_1 < \cdots <y_j$, 
then the link of $\sigma$ is equal to $\varDelta(x,y)$;
hence Proposition \ref{link sphere} yields that $\varDelta(x,y)$ is a homology $n-k-1$ sphere.
\end{proof}

If $\varDelta(P)$ is a homology manifold, 
then the following proposition tells us that the magnitude homology of $P$ is diagonal and torsion-free.

\begin{prop}[cf. \cite{Iva24}, Theorem 7.8]\label{manifold diagonality}
Let $(P,r)$ be a finite graded poset. 
If $\varDelta(P)$ is a homology $n$-manifold, then the magnitude homology of $P$ is diagonal and torsion free.
In particular,
\begin{equation}\label{mfd magnitude}
\MH_{i}^{k}(P) \cong \begin{cases}
\mathbb{Z}^{\#\{(x,y)\mid r(y)-r(x)=k\}} & (i=k),\\
0 & (i \neq k).
\end{cases}
\end{equation}
\end{prop}
\begin{proof}
Use Corollary \ref{interval} and \eqref{m decomposition}.
\end{proof}

The following proposition implies that the diagonality and torsion-freeness of magnitude homology can be collapsed simply by adding a maximum and a minimum elements.
Though it is a straightforward extension of \cite[Theorem 7.8]{IM24} and proof is essentially identical to that of \cite[Theorem 7.8]{IM24}; we include it for completeness.
\begin{prop}[cf.\cite{KY21} Corollary 5.12, \cite{Iva24} Theorem 7.8]\label{magnitude of manifold2}
Under the same assumption as above, 
let $\hat{P}$ be $P\cup \{\hat{0},\hat{1}\}$.
Then the magnitude homology of $\hat{P}$ is
\begin{equation}\label{mfd magnitude 2}
\MH^{k}_{i}(\hat{P})\cong \begin{cases}
\tilde{H}_{i-2}(\varDelta(P)) & (k=r(P) + 2),\\
\mathbb{Z}^{\#\{(x,y)\mid x\leq y, r(y)-r(x)=k\}} & (k\neq r(P) +2,\ i=k),\\
0 & (\mathrm{otherwise}).
\end{cases}
\end{equation}
\end{prop}
\begin{proof}
If $x,y\in \hat{P}$ are neither $\hat{0}$ nor $\hat{1}$, it is the same as Proposition \ref{mfd magnitude}.\\
For an open interval $(\hat{0},\hat{1})=P$, 
$\MH^{r(P)}_{i}(\hat{P})\cong \tilde{H}_{i-2}(\varDelta(\hat{0},\hat{1}))=\tilde{H}_{i-2}(\varDelta(P))$.\\
For any $x\in \hat{P}\setminus \{\hat{1}\}$, an open interval $(\hat{0},x)$ is isomorphic to $P_{< x}$.
Let $\sigma\in \varDelta(P)$ be a simplex indexed by a maximal chain $x<x_1<\cdots<x_j$ in $P_{\geq x}$, 
then $\varDelta(P_{<x}) = \mathrm{lk}_{\varDelta(P)}(\sigma)$ 
and Proposition \ref{link sphere} implies that $\varDelta(x,y)$ is a $r(y)-r(x)-2$ dimensional homology sphere.
The same applies for an open interval of type $(y,\hat{1})$.
Now Proposition \ref{decomposition} yields \eqref{mfd magnitude 2}.
\end{proof}

\subsection{Filtered order complexes of posets associated with homology manifolds}
\begin{thm}\label{closed manifold}
Let $(P,r)$ be a finite graded poset, and suppose that $\varDelta(P)$ is a triangulation of homology $n$-manifold $X$. 
Then for $0\leq k\leq n$, we have the following isomorphisms of the homology groups of $\varDelta^{(k)}(P)$, and the inequality for the $k$-th Betti number of $\varDelta^{(k)}(P)$: 
\begin{gather}
  H_{i}(\varDelta^{(k)}(P))\cong H_{i}(X) \quad (i\leq k-1),\\
n-k\leq \beta_{k}(\varDelta^{(k)}(P)) \leq \#\{(x,y)\mid x\leq y,\ r(y)-r(x)=k\} \label{mfd rank}.
\end{gather}
\end{thm}
\begin{proof}
$P$ is diagonal by Proposition \ref{mfd magnitude}, and Corollary \ref{diagonal order complex} implies that 
$H_{i}(\varDelta^{(k)}(P))\cong H_{i}(X)\ (i\leq k-1)$. 
Also, Proposition \ref{manifold diagonality} implies the right-hand side of the inequality \eqref{mfd rank}.
Take $\sigma=(x_i<x_{i+1}<\cdots<x_{i+k-1})\in \varDelta(P)$ so that the rank of an element $x_i$ is $i$, 
then $\mathrm{lk}\ \sigma\in \varDelta(P_{[i,i+k]})$ 
and Proposition \ref{link sphere} implies that $\mathrm{lk}\ \sigma$ is a nontrivial cycle of $H_{k}(\varDelta(P_{[i,i+k]}))$.
Proposition \ref{rank selection} implies that $\mathrm{lk}\ \sigma$ is also a nontrivial cycle of $H_{k}(\varDelta^{(k)}(P))$.
Given that $\varDelta(P_{[i,i+k]})\cap\varDelta(P_{[i+1,i+k+1]})=\varDelta(P_{[i+1,i+k]})$, if we take one $k$-dimensional cycle from each of  
$\varDelta(P_{[0,k]}),\ldots,\varDelta(P_{[n-k,n]})$, then these cycles are independent.
Thus, we obtain the left-hand side of the inequality \eqref{mfd rank}.
\end{proof}

\begin{cor}\label{closed manifold2}
Let $(P,r)$ be a graded poset, and suppose $\varDelta(P)$ gives a triangulation of a homology $n$-manifold $X$ without boundary.
Then for $0\leq k \leq n$, $\varDelta^{(k)}(P)$ is not contractible.
\end{cor}
\begin{proof}
$\varDelta^{(n)}(P)=\varDelta(P)$ is not contractible because if $X$ is orientable, $H_{n}(X)\cong \mathbb{Z}$ 
and if $X$ is not orientable, then $H_{n-1}(X)$ has a torsion subgroup of order $2$ (see, \cite[Corollary 65.5]{Mun84b}).
If $k\leq n-1$, Theorem \ref{closed manifold} implies that 
$\varDelta^{(k)}(P)$ is not contractible.
\end{proof}

\begin{rmk}
If $\varDelta(P)$ is a triangulation of a manifold with boundary, $\varDelta^{(k)}(P)$ can be contractible.
For example, let $P$ be a chain of length $n$, then $\varDelta(P)$ is a triangulation of a $n$-dimensional ball, 
and for $1\leq k \leq n$, $\varDelta^{(k)}(P)$ are all contractible by Example \ref{chain}.
\end{rmk}

The following proposition implies that even if the diagonality of the magnitude homology collapses by adding the maximam and minimum elements, 
the homology of subcomplexes have similar behaviors.
\begin{prop}
Let $(P,r)$ be a graded poset such that $\varDelta(P)$ gives a triangulation of a homology $n$-manifold.
For $1\leq k\leq n+1$, the homology groups of subcomplexes $\varDelta^{(k)}(\hat{P})$ are as follows:
\begin{align}\label{hat}
H_{i}(\varDelta^{(k)}(\hat{P})) & \cong \tilde{H}_{i-1}(X) & (1\leq i \leq k-1),\\
H_{0}(\varDelta^{(k)}(\hat{P}))& \cong \mathbb{Z} & (i=0).
\end{align}
\end{prop}
\begin{proof}
By Proposition \ref{relative homology} and Proposition \ref{mfd magnitude 2}, the long exact sequence of the homology of a pair of subcomplexes $(\varDelta^{(n+2)}(\hat{P}),\varDelta^{(n+1)}(\hat{P}))$, 
we have the following exact sequence of homology groups:
\begin{equation}
\cdots \rightarrow H_{i}(\varDelta^{(n+2)}(\hat{P}))\rightarrow \tilde{H}_{i-2}(X)\rightarrow H_{i-1}(\varDelta^{(n+1)}(\hat{P}))\rightarrow H_{i-1}(\varDelta^{(n+2)}(\hat{P}))\rightarrow \cdots
\end{equation}
Thus we obtain the isomorphism $\tilde{H}_{i-1}(X) \cong H_{i}(\varDelta^{(n+1)}(\hat{P}))$, since $\varDelta^{(n+2)}(\hat{P})=\varDelta(\hat{P})$ is contractible.
Also, we get $H_{i}(\varDelta^{(k)}(\hat{P}))\cong H_{i}(\varDelta^{(k+1)}(\hat{P}))\ (1\leq i \leq k-1)$ for $k\leq n$ by the long exact sequence of a pair $(\varDelta^{(k+1)}(\hat{P}),\varDelta^{(k)}(\hat{P}))$, 
and we have $H_{i}(\varDelta^{(k)}(\hat{P}))\cong H_{i-1}(X)$, hence \eqref{hat} holds.
\end{proof}

\begin{ex}
Let $P_n=\{a^{1}_{1},a^{1}_{2},\ldots,a^{1}_{n},a^{2}_{1},a^{2}_{2},\ldots,a^{2}_{n}\}$
with the partial order generated by $a^{p}_{i}<a^{q}_{i+1} \ (p,q\in \{1,2\})$.
$\varDelta(P_n)$ is a $n$-joins of $0$-sphere $\mathbb{S}^0$, which is homeomorphic to a $n$-sphere.
For $1\leq k \leq n$, 
$\varDelta^{(k)}(P_n)$ is homotopy equivalent to a wedge of $(2n-2k-1)$ copies of $\mathbb{S}^k$.

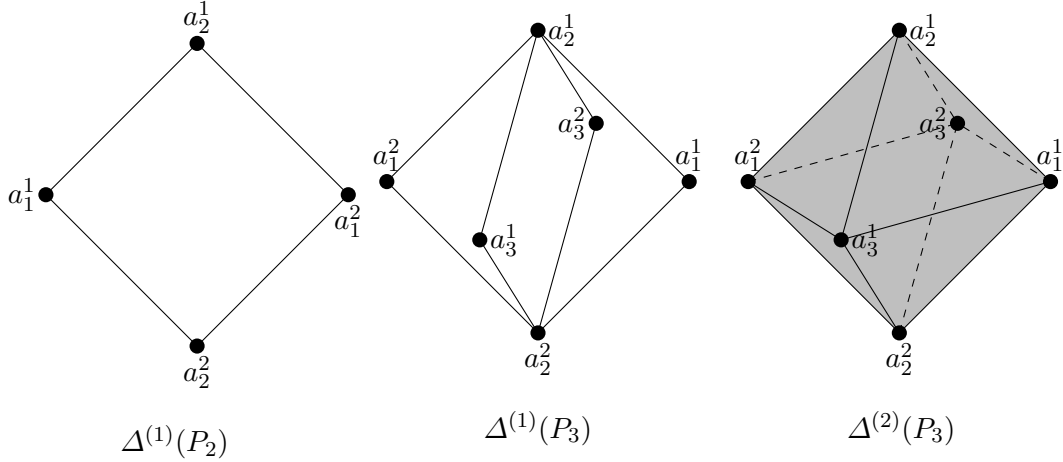
\begin{figure}[htbp]
\centering
  \begin{tabular}{cc}
    \begin{minipage}{0.3\textwidth}
      \centering

      \begin{tikzpicture}
        
        \coordinate (1) at (-2,0);
        \coordinate (-1) at (2,0);
        \coordinate (2) at (0,2);
        \coordinate (-2) at (0,-2);
        
        \node[left] at (1) {$a^{1}_{1}$};
        \node[below] at (-1) {$a^{2}_{1}$};
        \node[above] at (2) {$a^{1}_{2}$};
        \node[below] at (-2) {$a^{2}_{2}$};

        \draw (1)--(2)--(-1)--(-2)--cycle;

        \foreach\P in {1,-1,2,-2} \fill[black]
        (\P)circle(0.1);
      \end{tikzpicture}
      \caption*{$\varDelta^{(1)}(P_2)$}
    \end{minipage}
  
    \begin{minipage}{0.3\textwidth}
      \centering
      \begin{tikzpicture}
        \coordinate (1) at (2,0,0);
        \coordinate (-1) at (-2,0,0);
        \coordinate (2) at (0,2,0);
        \coordinate (-2) at (0,-2,0);
        \coordinate (3) at (0,0,2);
        \coordinate (-3) at (0,0,-2);

        \node[above] at (1) {$a^{1}_{1}$};
        \node[above] at (-1) {$a^{2}_{1}$};
        \node[right] at (2) {$a^{1}_{2}$};
        \node[below] at (-2) {$a^{2}_{2}$};
        \node[right] at (3) {$a^{1}_{3}$};
        \node[left] at (-3) {$a^{2}_{3}$};

        \draw (1)--(2)--(-1)--(-2)--cycle;
        \draw (2)--(3)--(-2)--(-3)--cycle;

        \foreach\P in {1,-1,2,-2,3,-3} \fill[black]
        (\P)circle(0.1);
      \end{tikzpicture}
      \caption*{$\varDelta^{(1)}(P_3)$}
    \end{minipage}

    \begin{minipage}{0.3\textwidth}
    \begin{tikzpicture}
        \coordinate (1) at (2,0,0);
        \coordinate (-1) at (-2,0,0);
        \coordinate (2) at (0,2,0);
        \coordinate (-2) at (0,-2,0);
        \coordinate (3) at (0,0,2);
        \coordinate (-3) at (0,0,-2);

        \fill[lightgray,draw=black] (1)--(2)--(-1)--(-2)--cycle;
        \draw (2)--(3)--(-2);
        \draw (1)--(3)--(-1);
        \draw[dashed] (2)--(-3)--(-2);
        \draw[dashed] (1)--(-3)--(-1);

        \node[above] at (1) {$a^{1}_{1}$};
        \node[above] at (-1) {$a^{2}_{1}$};
        \node[right] at (2) {$a^{1}_{2}$};
        \node[below] at (-2) {$a^{2}_{2}$};
        \node[right] at (3) {$a^{1}_{3}$};
        \node[left] at (-3) {$a^{2}_{3}$};

        \foreach\P in {1,-1,2,-2,3,-3} \fill[black]
        (\P)circle(0.1);
    \end{tikzpicture}
    \caption*{$\varDelta^{(2)}(P_3)$}
  \end{minipage}
  \end{tabular}
\end{figure}
\end{ex}
\begin{ex}
Let $K$ be a regular CW decomposition of a real projective space $\mathbb{R}P^2$ taken as Figure \ref{fig:RP2}.
Let $P$ be a face poset of $K$, then $\varDelta^{(2)}(P)$ is a triagulation of $\mathbb{R}P^2$, 
$\varDelta^{(1)}(P)$ is homotopy equivalent to a wedge of $12$ copies of $\mathbb{S}^1$.
\begin{figure}[H]
    \centering
    \begin{minipage}{0.45\textwidth}
        \centering
        % Left diagram (Fig. 7.1)
        \begin{tikzpicture}
            \coordinate (0) at (0,0) ;
            \coordinate (1) at (-2,0) ;
            \coordinate (2) at (0,2) ;
            \coordinate (3) at (2,0) ;
            \coordinate (4) at (0,-2) (4) ;

            \fill[lightgray, draw=black](0)--(1)--(2)--cycle;
            \fill[lightgray, draw=black](0)--(2)--(3)--cycle;
            \fill[lightgray, draw=black](0)--(3)--(4)--cycle;
            \fill[lightgray, draw=black](0)--(4)--(1)--cycle;

            \node[above right] at (0) {$0$};
            \node[left] at (1) {$1$};
            \node[above] at (2) {$2$};
            \node[right] at (3) {$3$};
            \node[below] at (4) {$4$};

            \draw[->, >=Stealth] (1) -- (4);
            \draw[->>, >=Stealth] (2) -- (1);
            \draw[->, >=Stealth] (3) -- (2);
            \draw[->>, >=Stealth] (4) -- (3);

            \foreach\P in {0,1,2,3,4} \fill[black]
            (\P)circle(0.1);
        \end{tikzpicture}
        \caption{$K$ is taken by identifying the opposite sites according to their orientation.}
        \label{fig:RP2}
      \end{minipage} \hspace{0.05\textwidth}
    \begin{minipage}{0.45\textwidth}
        \centering
        % Right diagram (Fig. 7.2)
        \begin{tikzpicture}
            \coordinate (0) at (0,0);
            \coordinate (1) at (-2,0);
            \coordinate (2) at (2,0);
            \coordinate (01) at (-3,2);
            \coordinate (12) at (-2,2);
            \coordinate (14) at (-1,2);
            \coordinate (04) at (1,2);
            \coordinate (03) at (2,2);
            \coordinate (02) at (3,2);
            \coordinate (014) at (-2.5,4);
            \coordinate (012) at (-1,4);
            \coordinate (034) at (1,4);
            \coordinate (023) at (2.5,4);

            \draw (0)--(01)--(014)--(14)--(1)--(12)--(012)--(01);
            \draw (0)--(04)--(034)--(12)--(2)--(02)--(023)--(03)--(0)--(02)--(012)--(12);
            \draw (01)--(1)--(03)--(034);
            \draw (023)--(14)--(2)--(04)--(014);

            \foreach\P in {0,1,2,01,12,14,04,03,02,014,012,034,023} \fill[black]
            (\P)circle(0.1);
        \end{tikzpicture}
        \caption{face poset of $K$}
    \end{minipage}
\end{figure}
\end{ex}

\section{Magnitude homology and filtered order complexes of shellable graded posets.}
\subsection{Shellablity of simplicial complexes and posets}
A simplicial complex $\varDelta$ is \textit{pure} or \textit{pure dimensional} of $n$ 
if all its maximal simplices have the same dimension $n$.

\begin{dfn}[\cite{Koz08}, Definition 12.1]\label{shellability}
A simplicial complex $\varDelta$ is called \textit{shellable} 
if its maximal simplices can be arranged in linear order $F_1,\ldots,F_t$ in such a way that 
the subcomplex $(\bigcup_{i=1}^{k-1}F_i)\cap F_{k}$ is pure and $(\mathrm{dim}F_k-1)$-dimensional for 
all $k=2,\ldots,t$. 
An ordering of maximal simplices satisfying the above condition is called a 
\textit{shelling order}. 
The maximal faces whose entire boundary is contained in 
the union of the earlier maximal simplices are called \textit{spanning simplices}.
\end{dfn}

The next theorem is one the most important topological properties of a shellable complex (we refer the reader to \cite[Theorem 12.3]{Koz08}).

\begin{thm}\label{shellable homotopy}
Let $\varDelta$ be a shellable simplicial complex, 
with $F_1,\ldots,F_t$ being the corresponding shelling order of the maximal simplices, 
and $\varSigma$ being the set of spanning simplices. 
Then $\varDelta$ is homotopy equivalent to $\bigvee_{\sigma\in \varSigma}\mathbb{S}^{\mathrm{dim}\ \sigma}$.
\end{thm}

\begin{rmk}
A poset $P$ is called a \textit{shellable poset} if its order complex $\varDelta(P)$ is shellable.
When $P$ has a maximal element and a minimal element, $P$ is shellable if and only if 
$\bar{P}=P\setminus \{\hat{0},\hat{1}\}$ is shellable.
\end{rmk}

The next result is important when considering the topology of the rank selection of a graded poset.

\begin{prop}[\cite{Koz08}, Proposition 12.6]\label{rank selection shellable}
Let $(P,r)$ be a shellable graded poset of rank $n$.
Then for any $S\subseteq \{0,\ldots,n\}$, the rank selection $P_S$ is also shellable 
and the shelling order for $\varDelta(P_{S})$ can be obtained from any shelling order for $\varDelta(P)$ 
by taking the restriction and then omittimg repetitions.
\end{prop}

\subsection{Filtered order complexes of shellable posets}
The next theorem states that the shellability of the order complex of a graded poset inherites to the subcomplexes.

\begin{thm}\label{shellable}
When a graded poset $(P,r)$ is shellable, $\varDelta^{(k)}(P)$ is also shellable.
\end{thm}
\begin{proof}
Let $F_1,\ldots,F_t$ be a shelling order for $\varDelta(P)$, 
and take $D^{0}_1,\ldots, D^{0}_{m_0}$ to be the sequence obtained by restricting $F_1,\ldots,F_t$ to $\varDelta(P_{[0,k]})$ 
and removing repetitions. 
By Proposition \ref{rank selection shellable}, $D^{0}_1,\ldots, D^{0}_{m_0}$ is a shelling order for $\varDelta(P_{[0,k]})$.
Let $D^{1}_1, \ldots, D^{1}_{m_1}$ be the sequence obtained by restricting $F_1,\ldots,F_t$ to $\varDelta(P_{[1,k+1]})$ and removing repetitions, 
and $i_j$ be the smallest number so that $D^{1}_{j}={F}_{i_j}\cap \varDelta(P_{[1,k+1]})$.
$(\bigcup_{i=1}^{m_0}D^{0}_i)\cap D^{1}_{1}=\varDelta(P_{[1,k]})\cap F_{i_1}$ is a pure $(k-1)=(\mathrm{dim}\ D^{1}_1 -1)$-dimensional subcomplex of $F_{i_1}$.
We have
\begin{equation}\label{shelling condition}
((\bigcup_{i=1}^{m_0}D^{0}_i)\cup (\bigcup_{i=1}^{j-1}D^{1}_i))\cap D^{1}_{j}
=((\varDelta(P_{[0,k]}))\cap D^{1}_{j})\cup ((\bigcup_{i=1}^{j-1}D^{1}_j)\cap D^{1}_{j}).
\end{equation}
Both $\varDelta(P_{[0,k]})\cap D^{1}_{j}=\varDelta(P_{[1,k]})\cap F_{i_j}$ and $(\bigcup_{i=1}^{j-1}D^{1}_j)\cap D^{1}_{j}$ are 
pure $(k-1)=(\mathrm{dim}\ D^{1}_{j}-1)$-dimensional subcomplexes of $D^{1}_{j}$, so that \eqref{shelling condition} is also a pure $(k-1)=(\mathrm{dim}\ D^{1}_{j}-1)$-dimensional subcomplex of $D^{1}_{j}$.
This implies that $D^{0}_1,\ldots, D^{0}_{m_0}, D^{1}_1, \ldots, D^{1}_{m_1}$ is a shelling order for $\varDelta(P_{[0,k]})\cup \varDelta(P_{[1,k+1]})$.
By applying the same method to $\varDelta^{(k)}(P)=\varDelta(P_{[0,k]})\cup \varDelta(P_{[1,k+1]}) \cup \cdots \cup \varDelta(P_{[n-k,n]})$, 
we obtain a shelling order for $\varDelta^{(k)}(P)$.
\end{proof}

\begin{cor}\label{cor:shellable}
For a shellable graded poset $(P,r)$ and $0\leq k \leq r(P)$, $\varDelta^{(k)}(P)$ is homotopy equivalent to a wedge of $k$-spheres, and the magnitude homology of $P$ is diagonal.
\end{cor}
\begin{proof}
Since all maximal simplices of $\varDelta^{(k)}(P)$ have dimension $k$, homotopy type of $\varDelta^{(k)}(P)$ follows by Theorem \ref{shellable homotopy} and Theorem \ref{shellable}.
By Proposition \ref{relative homology} and the long exact sequence of homology groups of a pair $(\varDelta^{(k)}(P),\varDelta^{(k-1)}(P))$, we have the following long exact sequence;
\begin{equation}
0\rightarrow H_k(\varDelta^{(k)}(P)) \rightarrow \MH^{k}_{k}(P) \rightarrow H_{k-1}(\varDelta^{(k-1)}(P)) \rightarrow \dots.
\end{equation}
Therefore, diagonality of the magnitude homology of $P$ holds by the homotopy type of $\varDelta^{(k)}(P)$.
\end{proof}

\begin{ex}
Let $W$ be a finite Coxeter group with generators $S$, and $T:=\{\alpha s \alpha^{-1}\mid s\in S, \alpha\in W\}$.
For $\sigma\in W$, the length $l(\sigma)$ of $\sigma$ is defined to be 
the minimum number $k$ such that there exist the expression 
\begin{equation}
\sigma=s_1\ldots s_k,
\end{equation}
where $s_i\in S$.
\textit{Bruhat order} on $W$ is a partial order relation on $W$ that is defined via the covering relation, $\sigma \prec \tau$ if
\begin{itemize}
\item $\tau=t\sigma$, for some $t\in T$,
\item $l(\tau)=l(\sigma)+1$. 
\end{itemize}
It is known that Bruhat order on $W$ is pure shellable, and that every open interval $(\sigma,\tau)$ of $W$ is homeomorphic to a $(l(\tau)-l(\sigma)-2)$ sphere (see, \cite[Chapter 2.7]{BB05}).
Thus, by Corollary \ref{cor:shellable}, $\varDelta^{(k)}(W)$ is homotopy equivalent to a wedge of $k$-spheres (including contractible case), and analogous to the closed manifold cases, the magnitude homology of $W$ is as follows:
\begin{equation}
\MH^{k}_{i}(W)\cong \begin{cases}
\mathbb{Z}^{\#\{(x,y)\mid l(y)-l(x)=k\}} & (k=i),\\
0 & (\mathrm{otherwise}).
\end{cases}
\end{equation}
\end{ex}

\begin{rmk}
Even if a graded poset $P$ has an order complex homotopy equivalent to a wedge of spheres, 
$\varDelta^{(k)}(P)$ is not always homotopy equivalent to a wedge of $k$-spheres, and also the magnitude homology of $P$ is not necessarily diagonal.
For example, let $P$ be a poset having a Hasse diagram obtained by adding the midpoints of edges of the Hasse diagram of Boolean algebra $\mathcal{B}_4$.
$\varDelta(P)$ is homotopy equivalent to a $2$-sphere, but $\varDelta^{(2)}(P)$ is not.
The magnitude homology of $P$ is as follows:
\begin{equation}
\MH^{4}_{i}(P)\cong \begin{cases}
\mathbb{Z}^{12} & (i=2)\\
0 & (\mathrm{otherwise}),
\end{cases}
\MH^{3}_{i}(P)=0,\ \MH^{2}_{i}(P)=0,
\MH^{1}_{1}(P)=\mathbb{Z}^{48}.
\end{equation}
The homology groups of $\varDelta^{(k)}(P)$ is as follows:
\begin{align}
H_{i}(\varDelta^{(4)}(P))\cong \begin{cases}
\mathbb{Z} & (i=2)\\
0 & (\mathrm{otherwise}),
\end{cases}
& \quad H_{i}(\varDelta^{(3)}(P))\cong \begin{cases}
\mathbb{Z}^{12} & (i=1)\\
0 & (\mathrm{otherwise}),
\end{cases}\\
H_{i}(\varDelta^{(2)}(P))\cong \begin{cases}
\mathbb{Z}^{12} & (i=1)\\
0 & (\mathrm{otherwise}),
\end{cases}& \quad
H_{i}(\varDelta^{(1)}(P))\cong \begin{cases}
\mathbb{Z}^{12} & (i=1)\\
0 & (\mathrm{otherwise}).
\end{cases}
\end{align}
\end{rmk}

\subsection{Filtered order complexes of geometric semilattices}
For a poset $P$ and its elements $x,y$, 
the upper bound of $x$ and $y$ is an element $u\in P$ such that $u\geq x,y$, 
and if there exists a minimal element of the upperbound of $x$ and $y$,
we call the element the \textit{join} of $x$ and $y$, which we denote $x\vee y$.
Dually, one can define the \textit{meet} of $x$ and $y$, which we denote $x\wedge y$.
If every pair of elements has a meet(respectively, join), then $P$ is called \textit{meet-semilattice}(or respectively, join-semilattice).
If every pair of elements has both meet and join, then $P$ is called \textit{lattice}.
An element of a finite lattice that covers $\hat{0}$ is called \textit{atom}. 

\begin{dfn}[\cite{WW86}, pp369]\label{dfn:geometric semilattice}
A lattice is called \textit{geometric} if its semimodular and atomistic(every element is a join of atoms).
A ranked meet-semilattice $L$ is called \textit{geometric semilattice} if $L$ satisfy the following conditions:
\begin{itemize}
\item every closed interval of $L$ is a geometric lattice.
\item for all $x \in L$ and subset $A$ of atoms of $L$ whose join exists, if $r(x)<r(\vee A)=|A|$, 
there is an atom $a\in A$ such that $a\nleq x$ and $a\vee x$ exists.
\end{itemize}
\end{dfn}

\begin{prop}[\cite{WW86}, Theorem 4.1]\label{upper ideal}
Let $L$ be a geometric semilattice, and $x\in L$. Then $L_{\geq x}:=\{y\in \mid x\leq y\}$ is a geometric semilattice.
\end{prop}

\begin{prop}[\cite{WW86}, Proposition 4.2]\label{truncation}
Any truncation of the top row ranks of a geometric semilattice is a geometric semilattice.
\end{prop}

\begin{prop}\label{join}
For any two elements $x,x'\in L$ with a element $u\in L$ such that $x\geq u$ and $x'\geq u$, their join $x\vee x'$ exists.
\end{prop}
\begin{proof}
By \cite[Theorem 3.2]{WW86}, there exists a geometric lattice $M$ and its atom $a\in M$ such that $L=M\setminus [a,\hat{1}]$. Since $M$ is a lattice, the join $x\vee x'$ exists in $M$.
If $u\geq x,x'$, then $u \geq x\vee x'$. As $u\ngeq a$, we must have $x\vee x'\ngeq a$, hence $x\vee x'\in L$.
\end{proof}
\begin{dfn}
Let $\mathcal{A}=\{H_1,\ldots,H_n\}$ be an arrangement of hyperplanes in $\mathbb{K}^{d}\ (\mathbb{K}=\mathbb{R},\mathbb{C})$.
The \textit{intersection poset} of $\mathcal{A}$ is a poset whose set of element is
\begin{equation}
\{K\subset \mathbb{K}^{d}\mid \exists I\subseteq \{1,\ldots,n\},\ \text{such that } \bigcap_{i\in I}H_i=K,\ K\neq \emptyset\}\cup \{\mathbb{K}^{d}\}
\end{equation}
with its order relation given by the reverse inclusion.
\end{dfn}

\begin{prop}[\cite{WW86} Proposition 3.1, \cite{OT92} Lemma 2.3]\label{arrangement}
The intersection poset ${L}(\mathcal{A})$ of a hyperplane arrangement $\mathcal{A}$ is a geometric semilattice.
\end{prop}
\begin{proof}
See \cite[Lemma 2.3]{OT92} for the fact that every interval is a geometric semilattice. 
Let $A=\{H_{i_1},\ldots,H_{i_k}\}$ be a set of atoms whose join exists, 
and set $H_{i_j}=\{(x_1,\ldots,x_d)\in \mathbb{K}^d\mid a_{j1}x_1+\cdots +a_{jd}x_d=b_j,\ a_{j1},\ldots,a_{jd},b_{j}\in \mathbb{K}\}$. 
then 
\begin{equation}
x\in H_{i_1}\vee \cdots \vee H_{i_k} \Leftrightarrow
\begin{pmatrix}
a_{11} & a_{12} & \dots & a_{1d}\\
a_{21} & a_{22} & \dots & a_{2d}\\
\vdots & \vdots & \ddots & \vdots\\
a_{k1} & a_{k2} & \dots & a_{kd}
\end{pmatrix}
\begin{pmatrix}
x_1\\
x_2\\
\vdots\\
x_d
\end{pmatrix}
=
\begin{pmatrix}
b_1\\
b_2\\
\vdots\\
b_d
\end{pmatrix}.
\end{equation}
This implies that a set of atoms is independent if and only if their normal vectors are linearly independent. 
Therefore, the second condition of the geometric semilattice follows by the exchange axiom of linear algebra.
\end{proof}

\begin{thm}[\cite{WW86} Theorem 7.2, Corollary 7.3]\label{semilattice}
Geometric semilattices are pure shellable.
\end{thm}

The above theorem combined with the following proposition 
yields that the order complex of the proper part of a geometric semilattice is homotopy equivalent to a nontrivial wedge of spheres.

\begin{prop}[\cite{Sta07}, Theorem 3.10]\label{Mobius}
Let $L$ be a geometric semilattice with the rank function $r$, and $\mu$ be M\"{o}bius function of $L$. 
For any $x\leq y \in L$, 
the following inequality follows:
\begin{equation}\label{nonzero}
(-1)^{r(y)-r(x)}\mu(x,y)>0.
\end{equation}
\end{prop}

\begin{prop}
Let $L$ be a geometric semilattice, and $x,y$ be elements of $L$ such that $r(y)-r(x)=k$. 
Then, $\varDelta(x,y)$ is homotopy equivalent to a nontrivial wedge of $k$-spheres.
\end{prop}
\begin{proof}
The open interval $(x,y)$ is a proper part of the closed interval $[x,y]$, which is a geometric lattice by the definition of a geometric semilattice. 
By Theorem \ref{semilattice}, $\varDelta(x,y)$ is a wedge of $k$-spheres and Proposition \ref{Mobius} guarentees the non-triviality since M\"{o}bius function $\mu(x,y)$ equals Euler characteristic.
\end{proof}

\begin{thm}\label{geometric semilattice}
Let $L$ be a geometric semilattice of rank $n$. 
Then $L$ is diagonal, and for $0\leq k \leq n-1$, $\varDelta^{(k)}(L)$ is homotopy equivalent to a nontrivial wedge of $k$-spheres.
Also, the following inequality holds for the top Betti number of $\varDelta^{(k)}(L)$:
  \begin{gather}
  \beta_{0}(\varDelta^{(0)}(L))=|L|,\\
  \beta_{n-1}(\varDelta^{(n-1)}(L))=\rk\MH^{n}_{n}(L), \label{sphere number 2}\\
\begin{split}
\sum_{i=0}^{n-k-1}
\max_{r(x)=i}
\{\rk(\MH(L^{\ge x}_{[i,i+k+1]}))\}
+\sum_{i=1}^{n-k}
\max_{r(x)=i-1,\ r(y)=i+k+1}
\{|\mu(x,y)|\}
\\
\le \beta_k(\varDelta^{(k)}(L))
\le \rk\MH^{k}_{k}(L)
\qquad (1\le k\le n-2)
\end{split}
\label{sphere number 1}
\end{gather}
\end{thm}
\begin{proof}
By Theorem \ref{shellable} and Theorem \ref{semilattice}, $\varDelta^{(k)}(L)$ is homotopy equivalent to a wedge of $k$-spheres, and \ref{diagonal order complex} implies the righthand side of the inequality \eqref{sphere number 1}.\\
Now we prove the lefthand side of the inequality \eqref{sphere number 1} and the equation \eqref{sphere number 2}.\\
When $k=n-1$, let $\{a_1,\ldots,a_m\}$ be a set of elements of rank $n$.
For a closed interval $[\hat{0},a_i]$ in $L$, 
$\varDelta^{(n-1)}([\hat{0},a_i])\subseteq \varDelta^{(n-1)}(L)$ is a suspension of $\varDelta(\hat{0},a_i)$,
and the subcomplexes $\varDelta^{(n-1)}([\hat{0},a_1]),\ldots, \varDelta^{(n-1)}([\hat{0},a_m])$ of $\varDelta^{(n-1)}(L)$ is glued together over a cone with the apex $\hat{0}$; 
hence
\begin{align}
\beta_{n-1}(\varDelta^{(n-1)}(L)) & =\sum_{x\in L,\ r(x)=n}\beta_{n-1}(\varDelta^{(n-1)}([\hat{0},x]))\\
 & = \sum_{i=1}^{m}\beta_{n-2}(\varDelta(\hat{0},a_i))\\
 & =\sum_{i=1}^{m}|\mu(\hat{0},a_i)|\\
 & = \MH^{n}_{n}(L),
\end{align}
and we have the equation \ref{sphere number 2}.\\
When $k \leq n-2$, for $0< i < n-k$,
let $x$ be an element of rank $i$ and $y$ be an element of rank $i+k+1$.
Then by Proposition \ref{Mobius}, 
$\varDelta(x,y)$ has nontrivial $k$-dimensional cycles contained in $\varDelta(L_{[i,i+k]})$, 
hence it is also a nontrivial $k$-dimensional cycle of $H_{k}(\varDelta^{(k)}(L))$ by Proposition \ref{relative homology}.
Also, let $x$ be a element of rank $i$, and set $L^{\geq x}_{[i,i+k+1]}:=\{y\in L \mid x\leq y,\ r(y)\in [i,i+k+1]\}$.
By Proposition \ref{upper ideal} and Proposition \ref{truncation}, $L^{\geq x}_{[i,i+k+1]}$ is a geometric semilattice, 
and $\varDelta^{(k)}(L^{\geq x}_{[i,i+k+1]})$ is a nontrivial $k$-cycle in $\varDelta^{(k)}(L)$. 
The $k$-the Betti number of $\varDelta^{(k)}(L^{\geq x}_{[i,i+k+1]})$ is $\rk(\MH(L^{\geq x}_{[i,i+k+1]}))$ by \eqref{sphere number 2}.
If $x',y'$ are elements of $L$ such that $r(x')=i,\ r(y')=i+k+2$, then $\varDelta^{(k)}(L^{\geq x}_{[i,i+k+1]})$ and $\varDelta(x',y')$ are both nontrivial $k$-cycles 
contained in $\varDelta(L_{[i,i+k]})\cup\varDelta(L_{[i+1,i+k+1]})$, and their intersection $\varDelta^{(k)}(L^{\geq x}_{[i,i+k+1]})\cap \varDelta(x',y')$ is empty or a cone with an apex $x\vee x'$ by Proposition \ref{join}.
Also, $(\varDelta(L_{[0,k]})\cup \cdots \cup \varDelta(L_{[i,i+k]}))\cap (\varDelta^{(k)}(L^{\geq x}_{[i,i+k+1]})\cup \varDelta(x',y'))=\varDelta(L_{i,i+k})\cap(\varDelta^{(k)}(L^{\geq x}_{[i,i+k+1]})\cup \varDelta(x',y'))$ 
is a union of a cone and a $(k-1)$-dimensional cycle, therefore the lefthand side of \eqref{sphere number 1} follows.
\end{proof}

\begin{cor}\label{proper part}
Let $L$ be a geometric semilattice of rank $r(L) \geq 2$, and $\bar{L}$ be its proper part. 
For $k\leq r(L)-2$, 
$\varDelta^{(k)}(\bar{L})$ is homotpy equivalent to a nontrivial wedge of $k$-spheres.
\end{cor}
\begin{proof}
For $k=0$, it is trivial since $L$ has at leasy 2 atoms.\\
Suppose $r(L)>2$.
For $1\leq k \leq r(L)-2$, choose a pair of elements $x,y \in L$ such that 
$x\leq y$ and $r(y)-r(x)=k+2$,
then $\varDelta(x,y)$ is a nontrivial $k$-dimensional cycle in $\varDelta^{(k)}(\bar{L})$ by the same argument as Theorem \ref{geometric semilattice}.
\end{proof}

The next corollary is a generalization of a result such that the Falkman complex of a rank $n$ hyperplane arrangement is homotopy equivalent to a nontrivial wedge of $(n-2)$ spheres (see, \cite[Theorem 4.109]{OT92}).

\begin{cor}
Let $\mathcal{A}$ be a hyperplane arrangement of rank $n$, and $L(\mathcal{A})$ be its intersection poset. Then both $L(\mathcal{A})$ and $\bar{L}(\mathcal{A})$ are diagonal, 
and for $0\leq k \leq n-1$, $\varDelta^{(k)}(L(\mathcal{A}))$ and $\varDelta^{(k)}(\bar{L}(\mathcal{A}))$ are both homotopy equivalent to 
nontrivial wedge of $k$-spheres.
\end{cor}
\begin{proof}
Since $L(\mathcal{A})$ is a geometric semilattice by Proposition \ref{arrangement}, 
the statement follows from Theorem \ref{geometric semilattice} and Corollary \ref{proper part}.
\end{proof}

\begin{ex}
  Let $\mathcal{B}_n$ be a Boolean lattice on $\{0,\ldots,n\}$, and $\varPi_n$ be a partition lattice on $\{1,\ldots,n\}$.
  For $\bar{\mathcal{B}_5}$ and $\varPi_4$, 
  the homotopy type of the filtered order complexes of these posets are as follows:
  \begin{align}
  \varDelta^{(3)}(\bar{\mathcal{B}}_5) &\simeq \mathbb{S}^3,
&\quad \varDelta^{(2)}(\bar{\mathcal{B}}_5) &\simeq \bigvee^{19} \mathbb{S}^{2},
&\quad \varDelta^{(1)}(\bar{\mathcal{B}}_5) &\simeq \bigvee^{41} \mathbb{S}^{1},\\
\varDelta^{(3)}(\varPi_4) &\simeq \ast,
&\quad \varDelta^{(2)}(\varPi_4)  &\simeq \bigvee^{6} \mathbb{S}^2,
&\quad H_{1}(\varDelta^{(1)}(\varPi_4)) &\simeq \bigvee^{17} \mathbb{S}^1.
  \end{align}
  Magnitude homologies of these posets are as follows:
  \begin{equation}
MH^{k}_{i}(\bar{\mathcal{B}_5}) \cong \begin{cases}
\mathbb{Z}^{20} & (k=i=3)\\
\mathbb{Z}^{60} & (k=i=2)\\
\mathbb{Z}^{70} & (k=i=1)\\
\mathbb{Z}^{30} & (k=i=0)\\
0 & (\mathrm{otherwise}),
\end{cases}
\qquad
MH^{k}_{i}(\varPi_4) \cong \begin{cases}
\mathbb{Z}^{6} & (k=i=3)\\
\mathbb{Z}^{23} & (k=i=2)\\
\mathbb{Z}^{31} & (k=i=1)\\
\mathbb{Z}^{15} & (k=i=0)\\
0 & (\mathrm{otherwise}).
\end{cases}
\end{equation}

\begin{figure}[htbp]
\begin{minipage}{0.95\linewidth}
    \centering
  \begin{tikzpicture}[
  scale=0.7,
  every node/.style={circle,fill=black,inner sep=1.2pt,draw=black}
]
%\centering
%----------------------------------------
% 左：proper part of B_5
%----------------------------------------
\begin{scope}
  % nodes
  \node (b1_1) at (-2.00,4.20) {}; % (1)
  \node (b1_2) at (-1.00,4.20) {}; % (2)
  \node (b1_3) at (0.00,4.20)  {}; % (3)
  \node (b1_4) at (1.00,4.20)  {}; % (4)
  \node (b1_5) at (2.00,4.20)  {}; % (5)

  \node (b2_1)  at (-4.50,2.80) {}; % (1,2)
  \node (b2_2)  at (-3.50,2.80) {}; % (1,3)
  \node (b2_3)  at (-2.50,2.80) {}; % (1,4)
  \node (b2_4)  at (-1.50,2.80) {}; % (1,5)
  \node (b2_5)  at (-0.50,2.80) {}; % (2,3)
  \node (b2_6)  at (0.50,2.80)  {}; % (2,4)
  \node (b2_7)  at (1.50,2.80)  {}; % (2,5)
  \node (b2_8)  at (2.50,2.80)  {}; % (3,4)
  \node (b2_9)  at (3.50,2.80)  {}; % (3,5)
  \node (b2_10) at (4.50,2.80)  {}; % (4,5)

  \node (b3_1)  at (-4.50,1.40) {}; % (1,2,3)
  \node (b3_2)  at (-3.50,1.40) {}; % (1,2,4)
  \node (b3_3)  at (-2.50,1.40) {}; % (1,2,5)
  \node (b3_4)  at (-1.50,1.40) {}; % (1,3,4)
  \node (b3_5)  at (-0.50,1.40) {}; % (1,3,5)
  \node (b3_6)  at (0.50,1.40)  {}; % (1,4,5)
  \node (b3_7)  at (1.50,1.40)  {}; % (2,3,4)
  \node (b3_8)  at (2.50,1.40)  {}; % (2,3,5)
  \node (b3_9)  at (3.50,1.40)  {}; % (2,4,5)
  \node (b3_10) at (4.50,1.40)  {}; % (3,4,5)

  \node (b4_1) at (-2.00,0.00) {}; % (1,2,3,4)
  \node (b4_2) at (-1.00,0.00) {}; % (1,2,3,5)
  \node (b4_3) at (0.00,0.00)  {}; % (1,2,4,5)
  \node (b4_4) at (1.00,0.00)  {}; % (1,3,4,5)
  \node (b4_5) at (2.00,0.00)  {}; % (2,3,4,5)

  % edges
  \draw (b1_1) -- (b2_1);
  \draw (b1_1) -- (b2_2);
  \draw (b1_1) -- (b2_3);
  \draw (b1_1) -- (b2_4);

  \draw (b1_2) -- (b2_1);
  \draw (b1_2) -- (b2_5);
  \draw (b1_2) -- (b2_6);
  \draw (b1_2) -- (b2_7);

  \draw (b1_3) -- (b2_2);
  \draw (b1_3) -- (b2_5);
  \draw (b1_3) -- (b2_8);
  \draw (b1_3) -- (b2_9);

  \draw (b1_4) -- (b2_3);
  \draw (b1_4) -- (b2_6);
  \draw (b1_4) -- (b2_8);
  \draw (b1_4) -- (b2_10);

  \draw (b1_5) -- (b2_4);
  \draw (b1_5) -- (b2_7);
  \draw (b1_5) -- (b2_9);
  \draw (b1_5) -- (b2_10);

  \draw (b2_1) -- (b3_1);
  \draw (b2_1) -- (b3_2);
  \draw (b2_1) -- (b3_3);

  \draw (b2_2) -- (b3_1);
  \draw (b2_2) -- (b3_4);
  \draw (b2_2) -- (b3_5);

  \draw (b2_3) -- (b3_2);
  \draw (b2_3) -- (b3_4);
  \draw (b2_3) -- (b3_6);

  \draw (b2_4) -- (b3_3);
  \draw (b2_4) -- (b3_5);
  \draw (b2_4) -- (b3_6);

  \draw (b2_5) -- (b3_1);
  \draw (b2_5) -- (b3_7);
  \draw (b2_5) -- (b3_8);

  \draw (b2_6) -- (b3_2);
  \draw (b2_6) -- (b3_7);
  \draw (b2_6) -- (b3_9);

  \draw (b2_7) -- (b3_3);
  \draw (b2_7) -- (b3_8);
  \draw (b2_7) -- (b3_9);

  \draw (b2_8) -- (b3_4);
  \draw (b2_8) -- (b3_7);
  \draw (b2_8) -- (b3_10);

  \draw (b2_9) -- (b3_5);
  \draw (b2_9) -- (b3_8);
  \draw (b2_9) -- (b3_10);

  \draw (b2_10) -- (b3_6);
  \draw (b2_10) -- (b3_9);
  \draw (b2_10) -- (b3_10);

  \draw (b3_1) -- (b4_1);
  \draw (b3_1) -- (b4_2);

  \draw (b3_2) -- (b4_1);
  \draw (b3_2) -- (b4_3);

  \draw (b3_3) -- (b4_2);
  \draw (b3_3) -- (b4_3);

  \draw (b3_4) -- (b4_1);
  \draw (b3_4) -- (b4_4);

  \draw (b3_5) -- (b4_2);
  \draw (b3_5) -- (b4_4);

  \draw (b3_6) -- (b4_3);
  \draw (b3_6) -- (b4_4);

  \draw (b3_7) -- (b4_1);
  \draw (b3_7) -- (b4_5);

  \draw (b3_8) -- (b4_2);
  \draw (b3_8) -- (b4_5);

  \draw (b3_9) -- (b4_3);
  \draw (b3_9) -- (b4_5);

  \draw (b3_10) -- (b4_4);
  \draw (b3_10) -- (b4_5);

  % label
  \node[draw=none,fill=none,inner sep=0pt] at (0,-1) {$\bar{\mathcal{B}_5}=\mathcal{B}_5 \setminus \{\emptyset,[5]\}$};
\end{scope}

%----------------------------------------
% 右：partition lattice Π_4
%----------------------------------------
\begin{scope}[xshift=12cm]
  % nodes
  \node (p4_1) at (0.00,0.00) {}; % {{1},{2},{3},{4}}

  \node (p3_1) at (-2.50,1.80) {}; % {1},{2},{3,4}
  \node (p3_2) at (-1.50,1.80) {}; % {1},{4},{2,3}
  \node (p3_3) at (-0.50,1.80) {}; % {1},{3},{2,4}
  \node (p3_4) at (0.50,1.80)  {}; % {3},{4},{1,2}
  \node (p3_5) at (1.50,1.80)  {}; % {2},{4},{1,3}
  \node (p3_6) at (2.50,1.80)  {}; % {2},{3},{1,4}

  \node (p2_1) at (-3.00,3.60) {}; % {1},{2,3,4}
  \node (p2_2) at (-2.00,3.60) {}; % {1,2},{3,4}
  \node (p2_3) at (-1.00,3.60) {}; % {4},{1,2,3}
  \node (p2_4) at (0.00,3.60)  {}; % {3},{1,2,4}
  \node (p2_5) at (1.00,3.60)  {}; % {1,3},{2,4}
  \node (p2_6) at (2.00,3.60)  {}; % {2},{1,3,4}
  \node (p2_7) at (3.00,3.60)  {}; % {1,4},{2,3}

  \node (p1_1) at (0.00,5.40) {}; % {1,2,3,4}

  % edges
  \draw (p4_1) -- (p3_1);
  \draw (p4_1) -- (p3_2);
  \draw (p4_1) -- (p3_3);
  \draw (p4_1) -- (p3_4);
  \draw (p4_1) -- (p3_5);
  \draw (p4_1) -- (p3_6);

  \draw (p3_1) -- (p2_1);
  \draw (p3_1) -- (p2_2);
  \draw (p3_1) -- (p2_6);

  \draw (p3_2) -- (p2_1);
  \draw (p3_2) -- (p2_3);
  \draw (p3_2) -- (p2_7);

  \draw (p3_3) -- (p2_1);
  \draw (p3_3) -- (p2_4);
  \draw (p3_3) -- (p2_5);

  \draw (p3_4) -- (p2_2);
  \draw (p3_4) -- (p2_3);
  \draw (p3_4) -- (p2_4);

  \draw (p3_5) -- (p2_3);
  \draw (p3_5) -- (p2_5);
  \draw (p3_5) -- (p2_6);

  \draw (p3_6) -- (p2_4);
  \draw (p3_6) -- (p2_6);
  \draw (p3_6) -- (p2_7);

  \draw (p2_1) -- (p1_1);
  \draw (p2_2) -- (p1_1);
  \draw (p2_3) -- (p1_1);
  \draw (p2_4) -- (p1_1);
  \draw (p2_5) -- (p1_1);
  \draw (p2_6) -- (p1_1);
  \draw (p2_7) -- (p1_1);

  % label
  \node[draw=none,fill=none,inner sep=0pt] at (0,-1) {$\varPi_4$};
\end{scope}
\end{tikzpicture}
\end{minipage}
\end{figure}
\end{ex}

\subsection{Magnitude homologies of the intersection semilattices of complex hyperplane arrangements}
The following theorem is the well-known fact about the relation between the topology of the complement of a complex hyperplane arrangement 
and its intersection semilattice.
\begin{thm}[\cite{OT92}, Theorem 5.93]
Let $\mathcal{A}$ be a hyperplane arrangement 
on $\mathbb{C}^n$.
The integral homology of the complement $M_{\mathcal{A}}:=\mathbb{C}^n\setminus \bigcup\mathcal{A}$ of $\mathcal{A}$ 
is torsion-free, and its betti number has a formula as follows:
\begin{equation}\label{betti}
\beta_{k}(M_{\mathcal{A}})=\sum_{x\in L(\mathcal{A}),\ r(x)=k}|\mu(\hat{0},x)|.
\end{equation}
\end{thm}

The following theorem holds for the magnitude homology of the intersection lattice of a complex arrangement and the topology of the complement of restricted arrangements.
\begin{thm}\label{complement}
In the same settings above, the next equation holds for the degree of the diagonal components of the magnitude homology of $L(\mathcal{A})$:
\begin{equation}\label{betti sum}
\rk\MH^{k}_{k}(L(\mathcal{A}))=\sum_{X\in L(\mathcal{A})}\beta_{k}(M(\mathcal{A}^{X})).
\end{equation}
Here, $\mathcal{A}^{X}:=\{X\cap H \mid X\nsubseteq H,\ X\cap H=\emptyset \}$.
\end{thm}
\begin{proof}
By \eqref{betti}, the betti numbers of the complement of $\mathcal{A}^{X}$ in $\mathbb{C}^{\mathrm{dim}X}$
are as follows:
\begin{equation}
b_k(M_{\mathcal{A}^X})=\sum_{\substack{Y\in L(\mathcal{A}),\ Y\geq X\\ r(Y)-r(X)=k}}|\mu(X,Y)|
\end{equation}
Therefore, by the diagonality of the magnitude homology of $L(\mathcal{A})$,
Corollary \ref{Mobius} yields \eqref{betti sum}.
\end{proof}

\begin{rmk}
Let $L$ be a geometric lattice, and $\OS^{\ast}(L)$ be its Orlik-Solomon algebra.
It is known that the degree of $\OS^{k}(L)$ is equal to $\sum_{x\in L,\ r(x)=k}|\mu(\hat{0},x)|$ (see \cite[Corollary 7.10.3]{White92}) 
hence by the same argument as Theorem \ref{complement}, we have
\begin{equation}
\rk \MH^{k}_{k}(L)=\sum_{X\in L}\rk \OS^{k}(L_{\leq X}).
\end{equation} 
Combining with Theorem \ref{geometric semilattice}, we get $\sum_{X\in L}\rk \OS^{n-1}(L_{\leq X})\leq \sum_{X\in L}\rk \OS^{n}(L_{\leq X})$.
It would be interesting to understand whether this inequality is related to the unimodality of the coefficients of the characteristic polynomial of a matroid (see, \cite{AHK18}).
\end{rmk}

\section*{Acknowledgements}
I would like to thank Masahiko Yoshinaga for his constant guidance and discussions.
I also thank Yasuhiko Asao, Yosuke Kusano (he conjectured Corollary \ref{closed manifold2}), Ye Liu, Takuya Saito and my former supervisor Shigeru Takamura for their valuable advice and comments.

\end{document}